\documentclass{amsart}
\usepackage{amssymb,latexsym,enumerate,amscd}

\newtheorem{Thm}[equation]{Theorem}
\newtheorem{Prop}[equation]{Proposition}
\newtheorem{Lem}[equation]{Lemma}

\theoremstyle{remark}

\newtheorem*{Rem*}{Remark}
\theoremstyle{definition}

\newtheorem{Property}[equation]{Property}

\newtheorem*{Not*}{Notation}

\DeclareMathOperator{\rk}{rk}
\DeclareMathOperator{\tr}{tr}
\begin{document}

\date{%
Thu Aug 16 12:45:54 EDT 2007}

\title[Embedding group algebras into regular rings]
{Embedding group algebras into finite von Neumann regular rings}

\author[P. A. Linnell]{Peter A. Linnell}
\address{Department of Mathematics \\
Virginia Tech \\
Blacksburg \\
VA 24061-0123 \\
USA}
\email{linnell@math.vt.edu}
\urladdr{http://www.math.vt.edu/people/plinnell/}

\begin{abstract}
Let $G$ be a group and let $K$ be a field of characteristic zero.  We
shall prove that $KG$ can be embedded into a von Neumann
unit-regular ring.  In the course of the proof, we shall obtain a
result relevant to the Atiyah conjecture.
\end{abstract}


\maketitle

The main result of this paper is
\begin{Thm} \label{Tmain}
Let $G$ be a group and let $K$ be a field of characteristic zero.
Then $KG$ can be embedded into a finite self-injective
von Neumann $*$-regular ring.
\end{Thm}
This result has been stated in \cite[p.~217]{KhuranaLam05}.  However
the proof there depends on a result of \cite{Faith03}.  The proof
presented here will not use that result, and will also give a result
related to the Atiyah conjecture \cite[\S 10]{Lueck02}.  We shall
apply the techniques of ultrafilters and ultralimits, as used in
\cite{ElekSzabo04}.

Let $G$ be a group and let $\mathcal{U}(G)$ denote the algebra of
unbounded operators on $\ell^2(G)$ affiliated to the group von
Neumann algebra $\mathcal{N}(G)$ of $G$ \cite[\S 8.1]{Lueck02}.
Then $\mathcal{U}(G)$ is a finite von Neumann regular $*$-ring that
is left and right self-injective, and also unit-regular \cite[\S
2,3]{Berberian82}.  For $\alpha,\beta
\in \mathcal{U}(G)$, we have $\alpha^*\alpha + \beta^*\beta = 0$ if
and only if $\alpha = \beta = 0$ \cite[p.~151]{Berberian82}.
Furthermore there is a unique projection $e \in \mathcal{N}(G)$
(so $e = e^2 = e^*$) such that $\alpha
\mathcal{U}(G) = e\mathcal{U}(G)$ and we have the following useful
result.

\begin{Lem} \label{Luseful}
Let $G$ be a group and let
$\alpha,\beta \in \mathcal{U}(G)$.  Then $(\alpha\alpha^* +
\beta\beta^*)\mathcal{U}(G) \supseteq \alpha\mathcal{U}(G)$.
\end{Lem}
\begin{proof}
Write $U = \mathcal{U}(G)$ and let
$e \in U$ be the unique projection such that $\alpha U = eU$.
Then $(1-e)U = \{u \in U \mid \alpha^* u = 0\}$.
Let $f \in U$ be the unique projection such that $(\alpha \alpha^* +
\beta \beta^*)U = fU$.  Then $(1-f)U = \{u \in U \mid (\alpha\alpha^*
+ \beta\beta^*)u = 0\}$.  Therefore if $u \in (1-f)U$, we have
$(\alpha\alpha^* + \beta\beta^*)u = 0$, hence
\[
0 = u^*\alpha\alpha^*u + u^*\beta\beta^*u =
(\alpha^* u)^*(\alpha^* u) + (\beta^* u)^*(\beta^* u)
\]
and we deduce that $\alpha^*u = 0$.  Thus $u \in (1-e)U$ and we
conclude that $(1-f)U \subseteq (1-e)U$.  Therefore we may write $1-f
= (1-e)v$ for some $v \in U$.  Then if $w \in eU$,
we find that $(1-f)w =
v^*(1-e)w = 0$, consequently $w = fw$ and hence $w \in fU$.  Thus $eU
\subseteq fU$ and the result follows.
\end{proof}

Now given $\alpha \in \mathcal{N}(G)$, we may write
$\alpha = \sum_{g \in G} \alpha_g g$ with $\alpha_g
\in \mathbb{C}$, and then $\tr
\alpha = \alpha_1$ is the trace of $\alpha$.  Moreover for $\alpha \in
\mathcal{U}(G)$, there is a unique projection $e \in
\mathcal{N}(G)$ such that $\alpha \mathcal{U}(G)
= e\mathcal{U}(G)$, and then we set $\rk\alpha =
\tr e$, the rank of $\alpha$.  Then $\rk \colon \mathcal{U}(G) \to
[0,1]$ is a rank function
\cite[Definition, p.~226, \S 16]{Goodearl91}.  This means that rk
satisfies (a)--(d) of Property \ref{Prk} below; it also satisfies
Properties \ref{Prk}\eqref{aastar},\eqref{subgroup}.
\begin{Property}\label{Prk}
\begin{enumerate}[\normalfont(a)]
\item \label{rka}
$\rk(1) = 1$.
\item \label{rkb}
$\rk(\alpha\beta) \le \rk(\alpha),\rk(\beta)$ for all $\alpha,\beta
\in \mathcal{U}(G)$.
\item \label{rkc}
$\rk(e+f) = \rk(e)+\rk(f)$ for all orthogonal idempotents $e,f \in
\mathcal{U}(G)$.
\item \label{rkd}
$\rk(\alpha) > 0$ for all nonzero $\alpha \in \mathcal{U}(G)$.
\item \label{aastar}
$\rk(\alpha^*\alpha) =\rk(\alpha\alpha^*) = \rk(\alpha)$ for all
$\alpha \in \mathcal{U}(G)$.
\item \label{subgroup}
If $G\le H$ and $\alpha \in \mathcal{U}(G)$, then $\rk(\alpha)$ is
the same whether we view $\alpha \in \mathcal{U}(G)$ or $\alpha \in
\mathcal{U}(H)$.
\end{enumerate}
\end{Property}
There is also a well-defined dimension $\dim_{\mathcal{U}(G)}$ for
$\mathcal{U}(G)$-modules \cite[Theorem 8.29]{Lueck02} which satisfies
$\dim_{\mathcal{U}(G)} \alpha\mathcal{U}(G) = \rk(\alpha)$.

We need the following result for generating units in
$\mathcal{U}(G)$.  Recall that an ICC group is a group in which all
conjugacy classes except the identity are infinite, and that a
unitary element is an element $u$ of $\mathcal{U}(G)$ such that
$u^*u=1$ (equivalently $uu^* = 1$, because $\mathcal{U}(G)$ is a
finite von Neumann algebra).
\begin{Lem} \label{Lboundedgeneration}
Let $G$ be an ICC group, let $\alpha \in \mathcal{U}(G)$, and let $n$
be a positive integer.  Suppose $\rk(\alpha) \ge 1/n$.  Then there
exist unitary elements $\tau_1,\dots,\tau_n \in \mathcal{U}(G)$ such
that $\sum_{i=1}^n \tau_i\alpha\alpha^*\tau_i^{-1}$
is a unit in $\mathcal{U}(G)$.
\end{Lem}
\begin{proof}
Suppose $e$ and $f$ are projections in $\mathcal{U}(G)$
such that $e\mathcal{U}(G) \cong f\mathcal{U}(G)$.  Since
$\mathcal{U}(G)$ is a unit-regular ring, we see that
$(1-e)\mathcal{U}(G) \cong (1-f)\mathcal{U}(G)$ by \cite[Theorem
4.5]{Goodearl91}.  Therefore $e$ and
$f$, and also $1-e$ and $1-f$, are algebraically equivalent
projections and hence equivalent \cite[\S 5]{Berberian82}.  We deduce
that $e$ and $f$ are unitarily equivalent \cite[p.~69 and Exercise
17.12]{Berberian72}; this means that there is a unitary element $u
\in \mathcal{U}(G)$ such that $ueu^{-1} = f$.

Suppose now that $e,f$ are projections in $\mathcal{U}(G)$ with
$\tr (e) \le \tr (f)$.
Since $G$ is an ICC group, the center of $\mathcal{N}(G)$ is
$\mathbb{C}$.  Therefore by \cite[Theorem 8.22 and
Theorem 9.13(1)]{Lueck02}, two
finitely generated projective $\mathcal{U}(G)$-modules $P,Q$ are
isomorphic if and only if $\dim_{\mathcal{U}(G)}(P) =
\dim_{\mathcal{U}(G)}(Q)$.
Using \cite[Theorem 8.4.4(ii)]{KadisonRingrose97}, we see that
there is a finitely generated projective
$\mathcal{U}(G)$-module $P$ such that $\dim_{\mathcal{U}(G)} P =
\tr(f) - \tr(e)$ and then $e\mathcal{U}(G) \oplus P \cong
f\mathcal{U}(G)$.  From the previous paragraph, we deduce that there
is a unitary element $u \in \mathcal{U}(G)$ such that
$ueu^{-1}\mathcal{U}(G) \subseteq f\mathcal{U}(G)$.

Set $\beta = \alpha\alpha^*$.  Then $\rk\beta = \rk\alpha\alpha^* \ge
1/n$ by Property \ref{Prk}\eqref{aastar}.
Suppose $0 \le r \le 1/\rk\beta -1$ and we have chosen unitary
elements $\tau_1,\dots,\tau_r \in \mathcal{U}(G)$ such that
\[
\rk(\tau_1\beta\tau_1^{-1} + \dots + \tau_r\beta\tau_r^{-1}) =
r\rk(\beta);
\]
certainly we can do this for $r=0$.  Since $\mathcal{U}(G)$ is a
regular von Neumann $*$-algebra, there is a unique projection $f \in
\mathcal{U}(G)$ such that $(\tau_1\beta\tau_1^{-1} + \dots +
\tau_r\beta\tau_r^{-1})\mathcal{U}(G) = f\mathcal{U}(G)$.
Also $\rk(1-f) \ge \rk\beta$, hence there is a unitary element
$\tau_{r+1} \in \mathcal{U}(G)$ such that
$\tau_{r+1}\beta\tau_{r+1}^{-1} \in (1-f)\mathcal{U}(G)$, and then
\[
\rk(\tau_1\beta\tau_1^{-1} + \dots + \tau_{r+1}\beta\tau_{r+1}^{-1})
= (r+1)\rk(\beta)
\]
by Lemma \ref{Luseful}.

Now suppose $r > 1/\rk(\beta)-1$ and we have chosen unitary
elements $\tau_1,\dots,\tau_r \in \mathcal{U}(G)$ such that
\[
\rk(\tau_1\beta\tau_1^{-1} + \dots + \tau_r\beta\tau_r^{-1})
\ge 1-\rk(\beta);
\]
by the previous paragraph, we can certainly do this if $r\le
1/\rk(\beta)$.  Again, let $f \in \mathcal{U}(G)$ be the unique
projection such that $(\tau_1\beta\tau_1^{-1} + \dots +
\tau_r\beta\tau_r^{-1})\mathcal{U}(G) = f\mathcal{U}(G)$.
Then $\rk(\beta) \ge \tr(1-f)$ and therefore there is a projection $e
\in \mathcal{U}(G)$ such that $1-f \in e\mathcal{U}(G)$ and $\tr(e) =
\rk(\beta)$.  Then we may choose a unitary element $\tau_{r+1} \in
\mathcal{U}(G)$ such that
$\tau_{r+1}\beta\tau_{r+1}^{-1}\mathcal{U}(G) =
e\mathcal{U}(G)$.  Applying Lemma \ref{Luseful}, we see that
\[
\rk(\tau_1\beta\tau_1^{-1} + \dots +
\tau_{r+1}\beta\tau_{r+1}^{-1}) = 1,
\]
which tells us that $\tau_1\beta\tau_1^{-1} + \dots +
\tau_{r+1}\beta\tau_{r+1}^{-1}$ is a unit in $\mathcal{U}(G)$.  We
deduce that $\tau_1\beta\tau_1^{-1} + \dots +
\tau_n\beta\tau_n^{-1}$ is a unit in $\mathcal{U}(G)$ as required.
\end{proof}

\begin{proof}[Proof of Theorem \ref{Tmain}]
We shall use the techniques of \cite[\S 2]{ElekSzabo04}.
By \cite[Theorem 1]{Scott51}, we may embed $G$ in a group which is
algebraically closed, so we may assume that $G$ is algebraically
closed.  It now follows from \cite[Corollary 1]{BHPS76}
that the augmentation ideal $\omega (KG)$ is the only
proper two-sided ideal of $KG$.  Let $\mathcal{F}$
denote the set of all finitely generated subfields of $K$.  For each
$F \in \mathcal{F}$, let $\mathcal{C}(F)$ denote the set of all
finitely generated subfields containing $F$: this is a subset of
$\mathcal{F}$.  Now let
\[
\mathcal{C} = \{X \subseteq \mathcal{F} \mid X \supseteq
\mathcal{C}(F) \text{ for some } F \in \mathcal{F}\}.
\]
The following are clear:
\begin{itemize}
\item
$\mathcal{C} \ne \emptyset$.

\item
If $A\subseteq B \subseteq \mathcal{F}$ and $A \in \mathcal{C}$, then
$B \in \mathcal{C}$ (if $\mathcal{C}(F) \subseteq A$, then
$\mathcal{C}(F) \subseteq B$).

\item
If $A,B \in \mathcal{C}$, then $A\cap B \in \mathcal{C}$ (if
$\mathcal{C}(E) \subseteq A$ and $\mathcal{C}(F) \subseteq B$, then
$\mathcal{C}(\langle E,F \rangle) \subseteq A \cap B$).
\end{itemize}
This means that the sets $\mathcal{C}$ form a filter, and hence they
are contained in an ultrafilter $\mathcal{D}$,
that is a maximal filter.
Thus $\mathcal{D}$ has the properties of $\mathcal{C}$ listed above,
and the additional property
\begin{itemize}
\item
If $X \subseteq \mathcal{F}$, then either $X$ or $\mathcal{F}
\setminus X$ is in $\mathcal{D}$.
\end{itemize}

Now set $R = \prod_{f \in \mathcal{F}} \mathcal{U}(G)$, the Cartesian
product of the $\mathcal{U}(G)$ (so infinitely many coordinates of an
element of $R$ may be nonzero).  Since the Cartesian product of
self-injective von Neumann unit-regular rings
is also self-injective von Neumann unit-regular, we see
that $R$ is also a self-injective von Neumann unit-regular
ring.  The general element of $\alpha \in R$ has coordinates
$\alpha_f$ for $f \in \mathcal{F}$, and we define $\rho(\alpha) =
\lim_{\mathcal{D}} \rho(\alpha_f)$, where lim indicates the limit of
$\rho(\alpha_f)$ associated to the ultrafilter $\mathcal{D}$.  Thus
$\rho (\alpha)$ has the property that it is the unique number which
is in the closure of $\{\rk(\alpha_d) \mid d \in X\}$ for all $X \in
\mathcal{D}$.  It is easy to check that $\rho$ is a pseudo-rank
function \cite[Definition, p.~226, \S 16]{Goodearl91}; thus $\rho$
satisfies Properties \ref{Prk}\eqref{rka}--\eqref{rkc}, but not
Property \ref{Prk}\eqref{rkd}, because
we can have $\rho(\alpha) = 0$ with $\alpha \ne 0$.

Let $I = \{r \in R \mid \rho(r) = 0\}$.  Using \cite[Proposition
16.7]{Goodearl91}, we see that $I$ is a two-sided ideal of $R$ and
$\rho$ induces a rank function on $R/I$.  Of course, $R/I$ will also
be a von Neumann unit-regular $*$-ring.  Next we show that $R/I$ is
self injective.  In view of \cite[Theorem 9.32]{Goodearl91}, we need
to prove that $I$ is a maximal ideal of $R$; of course this will
also show that $R/I$ is a simple ring.  Suppose $\alpha \in R
\setminus I$.  Then we may choose a real number $\epsilon$ such that
$0 < \epsilon < \rho(\alpha)$.  Set $S = \{s \in \mathcal{F} \mid
\rk(\alpha_s) < \epsilon\}$.  Then $\rho(\alpha)$ is not in the
closure of $\{\rk(\alpha_s) \mid s\in S\}$ and therefore $S \notin
\mathcal{D}$.  Let $T = \mathcal{F} \setminus S$, so $T \in
\mathcal{D}$ and choose a positive integer $n$ such that $n >
1/\epsilon$.  Since $G$ is an ICC group,
for each $t \in T$ there exist by Lemma
\ref{Lboundedgeneration} units $\tau(t)_1, \dots, \tau(t)_n \in
\mathcal{U}(G)$ such that
\[
\tau(t)_1 \alpha_t \tau(t)_1^{-1} + \dots +\tau(t)_n \alpha_t
\tau(t)_n^{-1}
\]
is a unit in $\mathcal{U}(G)$
(the important thing here is that $n$ is independent of $t$).
Now for $r = 1,\dots,n$,
define $\tau_r \in R$ by $(\tau_r)_t = \tau(t)_r$ for $t \in T$
and $(\tau_r)_s = 1$ for $s \in S$.  Then
\[
(\tau_1 \alpha \tau_1^{-1} + \dots + \tau_n \alpha \tau_n^{-1})_t
\]
is a unit for all $t \in T$ and it follows that its image in $R/I$ is
a unit.  This proves that $R/I$ is a simple ring.

Now we want to embed $KG$ into
$R/I$.  For each $f \in \mathcal{F}$, choose an embedding of $f$ into
$\mathbb {C}$.  This will in turn induce an embedding $\theta_f$ of
$fG$ into $\mathcal{U}(G)$.

For $\alpha \in KG$ and $f \in \mathcal{F}$, we define $\alpha_f =
\theta_f(\alpha)$ if $\alpha \in fG$, and $\alpha_f = 0$ otherwise.
This yields a well-defined map (not a homomorphism) $\phi\colon
KG \to R$ which in turn induces a map $\psi \colon KG \to R/I$.
Let $\alpha,\beta \in KG$ and let $F$ denote the
subfield of $K$ generated by the supports of $\alpha$ and $\beta$, so
$F \in \mathcal{F}$.  Since
\[
\alpha_f + \beta_f - (\alpha + \beta)_f = 0 =
\alpha_f\beta_f - (\alpha\beta)_f
\]
for all $f \in \mathcal{C}(F)$, we see that
\[
\rk(\alpha_f + \beta_f - (\alpha + \beta)_f)
= 0 =
\rk(\alpha_f\beta_f - (\alpha\beta)_f)
\]
for all $f \in \mathcal{C}(F)$.  Since $\mathcal{C}(F) \in
\mathcal{D}$, we deduce that $\psi(\alpha) + \psi(\beta) -
\psi(\alpha + \beta) = 0 = \psi(\alpha) \psi(\beta) -
\psi(\alpha\beta)$ and hence $\psi$ is a ring homomorphism.

Finally we show that $\ker\psi = 0$.  Since $\omega(KG)$ is the only
proper ideal of $KG$, we see that if $\ker\psi \ne 0$, then $g-1
\in \ker\psi$ for $1 \ne g \in G$.  But $\rk((g-1)_f)$ is a constant
positive number for $f \in \mathcal{F}$.  Therefore $\rho(g-1) \ne
0$ and hence $g-1 \notin \ker\psi$.  We conclude that
$\ker\psi = 0$ and the proof is complete.
\end{proof}

If $K$ is a field, $G$ is a group and $\theta$ is an automorphism of
$K$, then $\theta$ induces an automorphism $\theta_*$ of $KG$
by setting $\theta_*(\sum_g a_gg) = \sum_g \theta(a_g)g$.
It is not difficult to deduce from the proof of Theorem \ref{Tmain}
the following result related to the Atiyah conjecture \cite[\S
10]{Lueck02}.
\begin{Prop} \label{PAtiyah}
Let $G$ be a group, let $K$ be a subfield of $\mathbb{C}$ and let
$0 \ne \alpha \in KG$.
Then there exists $\epsilon > 0$ such that $\rk(\theta_*\alpha) >
\epsilon$ for every automorphism $\theta$ of $K$.
\end{Prop}
However we shall give an independent proof.
It ought to be true that $\rk(\theta_*\alpha) = \rk(\alpha)$ for
every automorphism $\theta$ of $K$.
\begin{proof}[Proof of Proposition \ref{PAtiyah}]
Using \cite[Theorem 1]{Scott51}, we may embed $G$ in a group which
is algebraically closed and has an element of infinite order.  Thus
by Property \ref{Prk}\eqref{subgroup}, we
may assume that $G$ is algebraically closed and contains an element
$x$ of infinite order.  Consider the two-sided ideal generated
by $\alpha(x-1)$.  Since $0 \ne \alpha(x-1) \in \omega(KG)$
and $\omega(KG)$ is the only proper two-sided ideal of $KG$
by \cite[Corollary 1]{BHPS76}, we see that
there exists a positive integer $n$ and $\beta_i,\gamma_i \in KG$
such that $\sum_{i=1}^n \beta_i\alpha\gamma_i = x-1$.  Then
\[
\sum_{i=1}^n \theta(\beta_i) \theta(\alpha) \theta(\gamma_i) = x-1.
\]
Since $\rk(x-1) = 1$ and $\rk(\theta(\beta_i)
\theta(\alpha) \theta(\gamma_i)) \le \rk(\theta\alpha)$ for all $i$,
we see that $\rk(\theta\alpha) \ge 1/n$ and the result follows.
\end{proof}

\bibliographystyle{plain}

\end{document}